\documentclass[11pt]{amsart}

\usepackage[T1]{fontenc}
\usepackage[utf8]{inputenc}
\usepackage{newtxtext,newtxmath}
\usepackage{mathtools}
\usepackage{enumitem}
\usepackage{xcolor}
\usepackage{microtype}
\usepackage[
  colorlinks=true,
  linkcolor=red,
  citecolor=green!55!black,
  urlcolor=blue
]{hyperref}

\allowdisplaybreaks
\setlist[enumerate]{label=(\roman*),leftmargin=2.6em}

\newtheorem{theorem}{Theorem}[section]
\newtheorem{proposition}[theorem]{Proposition}
\newtheorem{lemma}[theorem]{Lemma}
\newtheorem{corollary}[theorem]{Corollary}
\newtheorem{remark}[theorem]{Remark}
\theoremstyle{definition}

\newcommand{\C}{\mathbb C}
\newcommand{\N}{\mathbb N}
\newcommand{\rad}{\operatorname{rad}}
\newcommand{\supp}{\operatorname{supp}}
\newcommand{\wtB}{\widetilde B}
\newcommand{\inv}{\operatorname{inv}}

\title[Completion of a Continuous Inverse Algebra]{The Completion of a Continuous Inverse Algebra Need Not Be a Continuous Inverse Algebra}
\author{Zongjian Han}

\newcommand{\authoraffiliation}{School of Mathematical Sciences, Tongji University, Shanghai 200092, China}
\newcommand{\authoremail}{dbln@tongji.edu.cn}
\makeatletter
\renewcommand{\@setauthors}{%
  \begingroup
  \def\thanks{\protect\thanks@warning}%
  \trivlist
  \centering
  \@topsep30\p@\relax
  \advance\@topsep by -\baselineskip
  \item\relax
  \author@andify\authors
  \def\\{\protect\linebreak}%
  {\normalsize\authors\par}%
  \vspace{7pt}%
  {\footnotesize\normalfont\authoraffiliation\par}%
  \vspace{2pt}%
  {\footnotesize\normalfont E-mail:\ \href{mailto:\authoremail}{\texttt{\authoremail}}\par}%
  \ifx\@empty\contribs
  \else
    ,\penalty-3\space\@setcontribs
    \@closetoccontribs
  \fi
  \endtrivlist
  \endgroup
}
\makeatother
\date{August 23, 2026}

\begin{document}

\begin{abstract}
In 2006, Neeb asked whether the Hausdorff completion of a continuous inverse algebra (CIA) must again be a CIA.  The problem remained open in the general noncommutative case and was restated as Problem~10.6.6 in Gl\"ockner and Neeb's 2026 monograph, where the commutative case was proved affirmative.  We give a negative answer after twenty calendar years.  More sharply, we construct a Hausdorff metrizable locally $m$-convex complex CIA $B$ whose completion $\widehat B$ is a Fr\'echet locally $m$-convex algebra and for which inversion on $\widehat B^{\times}$ remains continuous, although $\widehat B^{\times}$ is not open.  Thus completion destroys precisely the local spectral stability at the identity that underlies the natural Lie-group structure of a CIA unit group.

The construction starts from Dixon's nonzero Jacobson-semisimple Banach algebra $C$ with a dense nil subalgebra $N$.  Endow the finite-support algebra $D:=N^{(\N)}$ with the topology inherited from the product $N^{\N}$ and set $B:=\C1\oplus D$.  Finite support makes every element of $D$ nilpotent, while Banach-algebra inversion in the unitization of $C$ yields continuity of inversion on $B$ without any locally uniform bound on nilpotence indices.  The completion is $\widehat B\cong\C1\oplus C^{\N}$.  A fixed $z\in C$ for which $1+z$ is noninvertible, shifted successively to coordinates tending to infinity, gives noninvertible elements converging to $1$.  The counterexample is necessarily noncommutative and identifies the precise boundary of the known commutative completion theorem.
\end{abstract}

\maketitle

\section{Introduction}

\subsection{Continuous inverse algebras and completion}

Continuous inverse algebras arose in the 1950s and form one of the basic sources of infinite-dimensional linear Lie groups; see the historical discussion preceding \cite[Definition~5.1.3]{GN2026}.  A unital locally convex algebra $A$ is a continuous inverse algebra, abbreviated CIA, if its unit group $A^{\times}$ is open and the inversion map
\[
\inv_A:A^{\times}\longrightarrow A,
\qquad
 a\longmapsto a^{-1},
\]
is continuous.  Because $A^{\times}$ is an open subset of the locally convex space $A$, these two conditions give $A^{\times}$ its natural locally convex Lie-group structure \cite[Example~5.1.4]{GN2026}.

The openness condition is a local spectral-stability statement: every sufficiently small perturbation of the identity must remain invertible.  Hausdorff completion, on the other hand, is the canonical analytic operation that replaces a locally convex algebra by a complete one, and continuous multiplication extends automatically to that completion.  It is therefore natural to ask whether the spectral and Lie-theoretic structure encoded by the CIA axioms also survives.  An affirmative theorem would permit one to pass from dense algebraic or smooth models to complete algebras without losing the corresponding unit Lie groups.  A particularly informative negative theorem should exhibit an obstruction created only at the completion boundary and invisible in the original dense algebra.

\subsection{A twenty-year problem}

Neeb posed the completion problem in print in 2006 as Problem~VIII.1: determine whether the completion of a CIA is again a CIA, or construct a counterexample \cite[Problem~VIII.1]{NeebSurvey}.  Gl\"ockner and Neeb restated the question in 2026 as Problem~10.6.6 and recorded the general case as open \cite[Problem~10.6.6]{GN2026}.  Its documented history therefore spans twenty calendar years, from 2006 to 2026.

The same 2026 monograph proves that the completion of every commutative CIA is again a CIA \cite[Proposition~10.6.10]{GN2026}.  Consequently, any counterexample must exploit a genuinely noncommutative mechanism.  For the negative result to identify the actual boundary of the theory, it is also important to keep the topology and multiplication as regular as possible and to determine which CIA axiom fails, rather than merely showing that the completion lies outside the class for an unspecified reason.

\subsection{Main theorem and its strength}

The required noncommutative input is supplied by Dixon \cite{Dixon1977}: there exists a nonzero complex Banach algebra $C$ such that
\[
\rad(C)=0
\]
and $C$ contains a dense nil subalgebra $N$.

\begin{theorem}[Completion counterexample]\label{thm:main}
There exists a Hausdorff metrizable locally $m$-convex complex CIA $B$ whose Hausdorff completion $\widehat B$ has the following properties:
\[
\begin{aligned}
&\widehat B\text{ is a Fr\'echet locally }m\text{-convex algebra},\\
&\inv_{\widehat B}:\widehat B^{\times}\longrightarrow\widehat B
  \text{ is continuous},\\
&\widehat B^{\times}\text{ is not open}.
\end{aligned}
\]
Consequently, $\widehat B$ is not a continuous inverse algebra.
\end{theorem}

This settles the completion problem negatively and is stronger than its bare negation.  Both algebras remain within the standard locally $m$-convex and metrizable framework; the completion is Fr\'echet; multiplication extends continuously; the completed unit group can be described exactly; and inversion remains continuous on that full unit group.  The sole failed axiom is openness:
\[
1\in\overline{\widehat B\setminus\widehat B^{\times}}.
\]
Thus the obstruction is not nonmetrizability, failure of completeness, discontinuity of multiplication, or discontinuity of inversion.  It is precisely a loss of local spectral stability at the identity.

\subsection{Construction and proof mechanism}

Set
\[
D:=N^{(\N)},
\qquad
B:=\C1\oplus D,
\]
where $D$ carries the topology inherited from the product $N^{\N}$.  The construction deliberately combines two features with opposite effects.  Finite support converts coordinatewise nilpotence into nilpotence of each entire sequence, and hence
\[
B^{\times}=\{\lambda1+d:\lambda\ne0\}.
\]
The product topology, however, allows a defect in a coordinate tending to infinity to escape every fixed neighborhood test.

The only delicate point before completion is continuity of inversion.  Nilpotence indices in $N$ need not be locally bounded, so a neighborhoodwise finite Neumann expansion is unavailable.  Instead, inversion in the Banach unitization $C^{\#}$ defines one continuous correction map
\[
\Theta(\lambda,n)
=(\lambda1+n)^{-1}-\lambda^{-1}1,
\]
which can be applied coordinatewise.  This proves directly that $B$ is a CIA without imposing any uniform nilpotence hypothesis.

Density of $N$ in $C$ then gives
\[
\widehat B\cong\C1\oplus C^{\N}
\]
with the product topology.  Jacobson semisimplicity supplies a fixed $z\in C$ such that $1+z$ is not invertible in $C^{\#}$.  If $e^{(k)}$ denotes the sequence having value $z$ in the $k$th coordinate and zero elsewhere, then
\[
e^{(k)}\longrightarrow0,
\qquad
1+e^{(k)}\notin\widehat B^{\times}
\quad(k\ge1).
\]
This proves non-openness.  At the same time, the exact coordinatewise unit criterion shows that inversion on $\widehat B^{\times}$ remains continuous, isolating the failure to one axiom.

Section~2 supplies the Banach-algebra input.  Sections~3 and~4 construct the finite-support CIA and establish continuity of its inversion.  Section~5 identifies the completion and its full unit group.  Section~6 proves that noninvertibility escapes to the identity.  Section~7 records the negative answer and the exact commutative--noncommutative boundary.

\section{The Banach-algebra input}

The later escape sequence requires one fixed element $z\in C$ whose perturbation of the identity is noninvertible.  We now derive this witness from Dixon's semisimple Banach algebra.

\subsection{Unitization and the radical witness}

Let $C$ be a complex Banach algebra, not assumed unital.  Its Banach unitization is
\[
C^{\#}:=\C1\oplus C
\]
with multiplication
\begin{equation}\label{eq:unitization-product}
(\lambda1+x)(\mu1+y)
=
\lambda\mu1+\lambda y+\mu x+xy
\end{equation}
and norm
\[
\|\lambda1+x\|_{\#}:=|\lambda|+\|x\|_C.
\]
The unitization supplies the identity used in all invertibility statements below.

We use the standard characterization of the Jacobson radical; see, for example, \cite[Chapter~II, Section~3]{Rickart1960}:
\begin{equation}\label{eq:radical-characterization}
\rad(C)
=
\left\{
 x\in C:
 1-yx\in(C^{\#})^{\times}
 \text{ for every }y\in C^{\#}
\right\}.
\end{equation}
Thus $\rad(C)=0$ means that every nonzero $x\in C$ admits a multiplier $y\in C^{\#}$ which destroys invertibility of $1-yx$.

\begin{lemma}[Noninvertible perturbation]\label{lem:bad-z}
Let $C$ be a nonzero complex Banach algebra satisfying $\rad(C)=0$.  Then there is an element $z\in C$ such that
\[
1+z\notin(C^{\#})^{\times}.
\]
\end{lemma}

\begin{proof}
Choose
\[
0\ne x\in C.
\]
Since $\rad(C)=0$, we have $x\notin\rad(C)$.  By \eqref{eq:radical-characterization}, there exists $y\in C^{\#}$ such that
\[
1-yx\notin(C^{\#})^{\times}.
\]
Because $C$ is a two-sided ideal of $C^{\#}$,
\[
yx\in C.
\]
Define
\[
z:=-yx\in C.
\]
Then
\[
1+z=1-yx\notin(C^{\#})^{\times}.
\]
\end{proof}

\subsection{Dixon's dense nil input}

An associative algebra $N$ is called nil if
\[
(\forall n\in N)(\exists r=r(n)\ge1)\qquad n^r=0.
\]
The exponent is allowed to depend on the element.  Nil is therefore strictly weaker than nilpotent.

\begin{theorem}[Dixon]\label{thm:dixon}
There exists a nonzero complex Banach algebra $C$ satisfying
\[
\rad(C)=0
\]
and containing a dense nil subalgebra
\[
N\subseteq C.
\]
\end{theorem}

\noindent\emph{Source.}
This is the construction of Dixon \cite{Dixon1977}.  In the present paper, $C$ and $N$ always denote one fixed pair supplied by Theorem~\ref{thm:dixon}.

The density of a nil algebra forces the ambient Banach algebra to be nonunital.  This elementary point ensures that the unitization $C^{\#}$ used above is the ordinary adjoining of a new identity.

\begin{lemma}[The Dixon algebra is nonunital]\label{lem:nonunital}
The Banach algebra $C$ in Theorem~\ref{thm:dixon} has no identity.
\end{lemma}

\begin{proof}
Assume that $C$ has an identity $1_C$.  Since $N$ is dense in $C$, there is $n\in N$ such that
\[
\|1_C-n\|_C<1.
\]
The Banach Neumann series gives
\[
n^{-1}
=
\sum_{j=0}^{\infty}(1_C-n)^j,
\]
so $n$ is invertible in $C$.  Since $N$ is nil, there exists $r\ge1$ such that
\[
n^r=0.
\]
Multiplication by $n^{-r}$ yields
\[
1_C=0,
\]
contrary to $C\ne0$.
\end{proof}

By Lemma~\ref{lem:bad-z}, we fix once and for all an element
\begin{equation}\label{eq:fixed-z}
z\in C,
\qquad
1+z\notin(C^{\#})^{\times}.
\end{equation}
This element is the noninvertibility witness placed into the escaping coordinates in Section~\ref{sec:escape}.

\section{The finite-support algebra}

The algebraic input $N$ is nil only elementwise.  To convert coordinatewise nilpotence into nilpotence of complete sequence elements, we now impose finite support.  To make defects escape under completion, we simultaneously retain the product topology.

\subsection{The product-subspace topology}

Define the finite-support algebra
\begin{equation}\label{eq:def-D}
D:=N^{(\N)}
=
\left\{
 d=(d_j)_{j\ge1}\in N^{\N}:
 \supp(d):=\{j:d_j\ne0\}\text{ is finite}
\right\}.
\end{equation}
Multiplication is coordinatewise:
\begin{equation}\label{eq:D-product}
(de)_j:=d_je_j
\qquad(j\ge1).
\end{equation}

The topology on $D$ is the subspace topology inherited from the product $N^{\N}$, where $N$ has the norm topology induced by $C$.  Equivalently, it is generated by
\begin{equation}\label{eq:D-seminorms}
r_m(d):=\max_{1\le j\le m}\|d_j\|_C
\qquad(m\ge1).
\end{equation}
This is not the locally convex direct-sum topology.  In particular, a neighborhood restricts only finitely many coordinates.

We adjoin a scalar identity and define
\begin{equation}\label{eq:def-B}
B:=\C1\oplus D.
\end{equation}
Its multiplication is
\begin{equation}\label{eq:B-product}
(\lambda1+d)(\mu1+e)
=
\lambda\mu1+\lambda e+\mu d+de.
\end{equation}
The locally convex topology on $B$ is generated by
\begin{equation}\label{eq:B-seminorms}
p_m(\lambda1+d)
:=
|\lambda|+r_m(d)
=
|\lambda|+\max_{1\le j\le m}\|d_j\|_C.
\end{equation}
The algebra $B$ is the candidate CIA whose completion will fail to be a CIA.

\subsection{Local \texorpdfstring{$m$}{m}-convexity}

The coordinatewise multiplication must be compatible with the topology in \eqref{eq:B-seminorms}.  The decisive estimate is submultiplicativity of every $p_m$.

\begin{proposition}[Locally $m$-convex structure]\label{prop:locally-m-convex}
For every $m\ge1$ and $a,b\in B$,
\[
p_m(ab)\le p_m(a)p_m(b).
\]
Consequently, $B$ is a Hausdorff metrizable locally $m$-convex algebra.
\end{proposition}

\begin{proof}
Write
\[
a=\lambda1+d,
\qquad
b=\mu1+e,
\]
and set
\[
D_m:=\max_{1\le j\le m}\|d_j\|_C,
\qquad
E_m:=\max_{1\le j\le m}\|e_j\|_C.
\]
By \eqref{eq:B-product},
\[
(ab)_j=\lambda e_j+\mu d_j+d_je_j.
\]
The Banach algebra inequality in $C$ gives
\begin{align*}
p_m(ab)
&=
|\lambda\mu|
+
\max_{1\le j\le m}
\|\lambda e_j+\mu d_j+d_je_j\|_C\\
&\le
|\lambda\mu|+|\lambda|E_m+|\mu|D_m+D_mE_m\\
&=
(|\lambda|+D_m)(|\mu|+E_m)\\
&=
p_m(a)p_m(b).
\end{align*}
Thus each $p_m$ is submultiplicative.  The countable family $(p_m)_{m\ge1}$ separates points, so $B$ is Hausdorff and metrizable.
\end{proof}

\subsection{Nilpotence of finite-support elements}

Finite support is used here for the first time.  It turns the coordinate-dependent nilpotence exponents in $N$ into one exponent for each element of $D$.

\begin{lemma}[Finite-support nilpotence]\label{lem:D-nil}
The algebra $D$ is nil.  More precisely, for every $d\in D$ there exists $r=r(d)\ge1$ such that
\[
d^r=0.
\]
\end{lemma}

\begin{proof}
Fix $d\in D$ and write
\[
F:=\supp(d).
\]
The set $F$ is finite by \eqref{eq:def-D}.  For every $j\in F$, the element $d_j\in N$ is nilpotent, so there exists $r_j\ge1$ such that
\[
d_j^{r_j}=0.
\]
Set
\[
r:=\max_{j\in F}r_j,
\]
with $r:=1$ if $F=\varnothing$.  Then for every $j\ge1$,
\[
(d^r)_j=d_j^r=0.
\]
Hence $d^r=0$.
\end{proof}

\subsection{The unit group}

Lemma~\ref{lem:D-nil} reduces invertibility in $B$ to the scalar coordinate.  This gives an open unit group before continuity of inversion is addressed.

\begin{proposition}[Units of the finite-support algebra]\label{prop:B-units}
The unit group of $B$ is
\begin{equation}\label{eq:B-units}
B^{\times}
=
\{\lambda1+d:\lambda\in\C^{\times},\ d\in D\}.
\end{equation}
In particular, $B^{\times}$ is open in $B$.
\end{proposition}

\begin{proof}
Define the continuous unital algebra homomorphism
\[
\chi:B\longrightarrow\C,
\qquad
\chi(\lambda1+d):=\lambda.
\]
If $\lambda1+d\in B^{\times}$, then
\[
\lambda=\chi(\lambda1+d)\in\C^{\times}.
\]
Thus invertibility implies $\lambda\ne0$.

Conversely, let $\lambda\ne0$.  By Lemma~\ref{lem:D-nil}, choose $r\ge1$ such that $d^r=0$.  Then
\begin{align}
(\lambda1+d)^{-1}
&=
\lambda^{-1}(1+\lambda^{-1}d)^{-1}\notag\\
&=
\lambda^{-1}
\sum_{\ell=0}^{r-1}
(-\lambda^{-1}d)^{\ell}
\in B.
\label{eq:finite-neumann}
\end{align}
This proves \eqref{eq:B-units}.  Finally,
\[
B^{\times}=\chi^{-1}(\C^{\times})
\]
is open because $\chi$ is continuous.
\end{proof}

\section{Continuity of inversion on the finite-support algebra}

Formula \eqref{eq:finite-neumann} is pointwise finite, but its truncation index $r(d)$ may be unbounded on every neighborhood.  It therefore does not by itself establish continuity of inversion.  We replace the varying polynomial formula by a single continuous map inherited from the Banach unitization $C^{\#}$.

\subsection{The coordinate inverse map}

For $\lambda\in\C^{\times}$ and $n\in N$, the element $\lambda1+n$ is invertible in $C^{\#}$ because $n$ is nilpotent.  Its inverse has scalar part $\lambda^{-1}$.  This motivates the correction map
\begin{equation}\label{eq:theta}
\Theta:\C^{\times}\times N\longrightarrow N,
\qquad
\Theta(\lambda,n)
:=(\lambda1+n)^{-1}-\lambda^{-1}1.
\end{equation}

\begin{lemma}[Continuity of the coordinate inverse]\label{lem:theta-continuous}
The map $\Theta$ in \eqref{eq:theta} is well-defined and continuous.
\end{lemma}

\begin{proof}
Fix $(\lambda,n)\in\C^{\times}\times N$.  Choose $r\ge1$ such that $n^r=0$.  Then
\begin{align*}
(\lambda1+n)^{-1}
&=
\lambda^{-1}
\sum_{\ell=0}^{r-1}(-\lambda^{-1}n)^{\ell},\\
\Theta(\lambda,n)
&=
\sum_{\ell=1}^{r-1}
(-1)^{\ell}\lambda^{-\ell-1}n^{\ell}
\in N.
\end{align*}
Thus $\Theta$ is well-defined with values in $N$.

The map
\[
\C^{\times}\times N\longrightarrow(C^{\#})^{\times},
\qquad
(\lambda,n)\longmapsto\lambda1+n,
\]
is continuous.  Since $C^{\#}$ is a Banach algebra, its inversion map is continuous.  Hence the composite
\[
(\lambda,n)
\longmapsto
(\lambda1+n)^{-1}-\lambda^{-1}1
\]
is continuous as a map into $C$.  Its image lies in $N$, and $N$ carries the subspace topology inherited from $C$.  Therefore its corestriction \eqref{eq:theta} is continuous.
\end{proof}

\subsection{The CIA property}

The product topology is characterized by coordinatewise continuity.  Applying Lemma~\ref{lem:theta-continuous} in every coordinate therefore gives the required global inversion map on $B$.

\begin{theorem}[The finite-support algebra is a CIA]\label{thm:B-CIA}
The algebra $B$ defined in \eqref{eq:def-B} is a Hausdorff metrizable locally $m$-convex continuous inverse algebra.
\end{theorem}

\begin{proof}
By Proposition~\ref{prop:locally-m-convex}, $B$ is a Hausdorff metrizable locally $m$-convex algebra.  By Proposition~\ref{prop:B-units}, its unit group is open.

Let
\[
a=\lambda1+d\in B^{\times}.
\]
Coordinatewise inversion in $C^{\#}$ gives
\begin{equation}\label{eq:B-inverse-coordinate}
a^{-1}
=
\lambda^{-1}1+
\bigl(\Theta(\lambda,d_j)\bigr)_{j\ge1}.
\end{equation}
Since $\Theta(\lambda,0)=0$, we have
\[
\supp\bigl((\Theta(\lambda,d_j))_{j\ge1}\bigr)
\subseteq\supp(d),
\]
so the sequence in \eqref{eq:B-inverse-coordinate} belongs to $D$.

For every fixed $j\ge1$, the map
\[
B^{\times}\longrightarrow N,
\qquad
\lambda1+d\longmapsto\Theta(\lambda,d_j),
\]
is continuous by Lemma~\ref{lem:theta-continuous}.  The universal property of the product topology gives continuity of
\[
B^{\times}\longrightarrow N^{\N},
\qquad
\lambda1+d\longmapsto
\bigl(\Theta(\lambda,d_j)\bigr)_{j\ge1}.
\]
Its image lies in $D$, and $D$ has the subspace topology from $N^{\N}$.  Hence the same map is continuous with codomain $D$.  Together with continuity of
\[
\lambda\longmapsto\lambda^{-1}
\quad\text{on }\C^{\times},
\]
formula \eqref{eq:B-inverse-coordinate} proves continuity of
\[
\inv_B:B^{\times}\longrightarrow B.
\]
Thus $B$ is a CIA.
\end{proof}

\begin{remark}[No uniform nilpotence is used]\label{rem:no-uniform}
The proof does not assert that the exponent $r(n)$ in $n^{r(n)}=0$ is locally bounded on $N$.  The finite Neumann formula is used only to show that $\Theta(\lambda,n)$ belongs to $N$ for each fixed $n$.  Continuity comes from Banach-algebra inversion in $C^{\#}$.
\end{remark}

\section{The completion and its unit group}

We now remove the finite-support condition by completion.  The ambient product algebra carries the same finite-coordinate seminorms, so it is both the natural completion of $B$ and the space in which noninvertibility can escape.

\subsection{The complete ambient algebra}

Define
\begin{equation}\label{eq:def-Btilde}
\wtB:=\C1\oplus C^{\N}.
\end{equation}
Its multiplication is the common-scalar coordinatewise product
\begin{equation}\label{eq:Btilde-product}
(\lambda1+c)(\mu1+e)
=
\lambda\mu1+
\bigl(\lambda e_j+\mu c_j+c_je_j\bigr)_{j\ge1}.
\end{equation}
The topology is generated by
\begin{equation}\label{eq:Btilde-seminorms}
q_m(\lambda1+c)
:=
|\lambda|+
\max_{1\le j\le m}\|c_j\|_C
\qquad(m\ge1).
\end{equation}

\begin{proposition}[Complete ambient algebra]\label{prop:Btilde-complete}
The algebra $\wtB$ is a Fr\'echet locally $m$-convex algebra.
\end{proposition}

\begin{proof}
The estimate in Proposition~\ref{prop:locally-m-convex}, with $d_j,e_j\in C$, gives
\[
q_m(ab)\le q_m(a)q_m(b)
\qquad(a,b\in\wtB).
\]
Thus $\wtB$ is locally $m$-convex and metrizable.

Let $(a^{(k)})_{k\ge1}$ be Cauchy for all seminorms $q_m$, and write
\[
a^{(k)}=\lambda_k1+(c_j^{(k)})_{j\ge1}.
\]
Since
\[
|\lambda_k-\lambda_\ell|
\le q_1(a^{(k)}-a^{(\ell)}),
\]
the scalar sequence $(\lambda_k)$ is Cauchy in $\C$ and converges to some $\lambda\in\C$.  For each fixed $j$,
\[
\|c_j^{(k)}-c_j^{(\ell)}\|_C
\le q_j(a^{(k)}-a^{(\ell)}),
\]
so $(c_j^{(k)})_k$ is Cauchy in the Banach space $C$ and converges to some $c_j\in C$.  Put
\[
a:=\lambda1+(c_j)_{j\ge1}\in\wtB.
\]
For every fixed $m$, convergence of the finitely many coordinates $1\le j\le m$ gives
\[
q_m(a^{(k)}-a)\longrightarrow0.
\]
Hence $a^{(k)}\to a$ in $\wtB$, proving completeness.
\end{proof}

\subsection{Identification of the completion}

The embedding $N\subseteq C$ induces a unital algebra embedding
\begin{equation}\label{eq:B-embedding}
\iota:B\longrightarrow\wtB,
\qquad
\iota(\lambda1+d):=\lambda1+d.
\end{equation}
The seminorm identities
\[
q_m(\iota(a))=p_m(a)
\]
show that $\iota$ is a topological embedding.

\begin{proposition}[Completion formula]\label{prop:completion-formula}
The image $\iota(B)$ is dense in $\wtB$.  Consequently,
\begin{equation}\label{eq:completion-isomorphism}
\widehat B\cong\wtB=\C1\oplus C^{\N}
\end{equation}
as complete locally convex algebras.
\end{proposition}

\begin{proof}
Fix
\[
a=\lambda1+(c_j)_{j\ge1}\in\wtB,
\qquad
m\ge1,
\qquad
\varepsilon>0.
\]
Since $N$ is dense in $C$, for each $1\le j\le m$ choose $n_j\in N$ such that
\[
\|c_j-n_j\|_C<\varepsilon.
\]
Define
\[
d^{(m)}:=(n_1,\ldots,n_m,0,0,\ldots)\in D.
\]
Then
\[
q_m\bigl(a-(\lambda1+d^{(m)})\bigr)
=
\max_{1\le j\le m}\|c_j-n_j\|_C
<\varepsilon.
\]
Thus every basic neighborhood of $a$ meets $\iota(B)$, so $\iota(B)$ is dense.  Proposition~\ref{prop:Btilde-complete} and the universal property of Hausdorff completion yield \eqref{eq:completion-isomorphism}.
\end{proof}

\subsection{Exact description of the completed units}

The completed unit group is determined coordinatewise, but all coordinate inverses must have the same scalar part.  For every $j\ge1$, define the continuous unital algebra homomorphism
\begin{equation}\label{eq:rho-j}
\rho_j:\wtB\longrightarrow C^{\#},
\qquad
\rho_j(\lambda1+c):=\lambda1+c_j.
\end{equation}

\begin{proposition}[Units of the completion]\label{prop:Btilde-units}
The unit group of $\wtB$ is
\begin{equation}\label{eq:Btilde-units}
\wtB^{\times}
=
\left\{
\lambda1+c:
\lambda\ne0,
\ \lambda1+c_j\in(C^{\#})^{\times}
\text{ for every }j\ge1
\right\}.
\end{equation}
\end{proposition}

\begin{proof}
Let $a=\lambda1+c\in\wtB^{\times}$.  The scalar projection
\[
\chi_{\infty}:\wtB\longrightarrow\C,
\qquad
\chi_{\infty}(\lambda1+c):=\lambda,
\]
is a unital algebra homomorphism, so
\[
\lambda=\chi_{\infty}(a)\in\C^{\times}.
\]
Moreover, \eqref{eq:rho-j} gives
\[
\rho_j(a)\in(C^{\#})^{\times}
\qquad(j\ge1).
\]
This proves necessity.

Conversely, assume that the conditions on the right-hand side of \eqref{eq:Btilde-units} hold.  The quotient homomorphism
\[
\pi:C^{\#}\longrightarrow\C,
\qquad
\pi(\alpha1+x):=\alpha,
\]
satisfies
\[
\pi\bigl((\lambda1+c_j)^{-1}\bigr)=\lambda^{-1}.
\]
Hence there is a unique $b_j\in C$ such that
\[
(\lambda1+c_j)^{-1}=\lambda^{-1}1+b_j.
\]
Set
\[
b:=(b_j)_{j\ge1}\in C^{\N}.
\]
By coordinatewise multiplication,
\[
(\lambda1+c)(\lambda^{-1}1+b)
=
(\lambda^{-1}1+b)(\lambda1+c)
=1.
\]
Thus $\lambda1+c\in\wtB^{\times}$.
\end{proof}

\subsection{Inversion remains continuous}

The counterexample will therefore not come from discontinuity of inversion.  We record this stronger property before proving non-openness.

Define the open set
\[
\Omega
:=
\left\{
(\lambda,x)\in\C\times C:
\lambda1+x\in(C^{\#})^{\times}
\right\}.
\]
Since the scalar quotient of an invertible element is invertible, every $(\lambda,x)\in\Omega$ satisfies $\lambda\ne0$.  Define
\begin{equation}\label{eq:Theta-C}
\Theta_C:\Omega\longrightarrow C,
\qquad
\Theta_C(\lambda,x)
:=(\lambda1+x)^{-1}-\lambda^{-1}1.
\end{equation}
The map $\Theta_C$ is continuous because inversion in the Banach algebra $C^{\#}$ is continuous.

\begin{proposition}[Continuous inversion after completion]\label{prop:Btilde-inversion}
The inversion map
\[
\inv_{\wtB}:\wtB^{\times}\longrightarrow\wtB
\]
is continuous.
\end{proposition}

\begin{proof}
For $a=\lambda1+c\in\wtB^{\times}$, Proposition~\ref{prop:Btilde-units} and \eqref{eq:Theta-C} give
\begin{equation}\label{eq:Btilde-inverse}
a^{-1}
=
\lambda^{-1}1+
\bigl(\Theta_C(\lambda,c_j)\bigr)_{j\ge1}.
\end{equation}
For each fixed $j$, the map
\[
\wtB^{\times}\longrightarrow C,
\qquad
\lambda1+c\longmapsto\Theta_C(\lambda,c_j),
\]
is continuous.  Therefore
\[
\lambda1+c
\longmapsto
\bigl(\Theta_C(\lambda,c_j)\bigr)_{j\ge1}
\]
is continuous as a map into the product $C^{\N}$.  Together with continuity of $\lambda\mapsto\lambda^{-1}$, formula \eqref{eq:Btilde-inverse} proves the claim.
\end{proof}

\section{Noninvertibility escaping to infinity}\label{sec:escape}

The completion $\wtB$ has continuous inversion on its unit group.  It remains to test openness.  The witness $z$ fixed in \eqref{eq:fixed-z} is placed in coordinates which eventually lie beyond every seminorm $q_m$.

For each $k\ge1$, define
\begin{equation}\label{eq:e-k}
e^{(k)}
:=(0,\ldots,0,\underset{k\text{th}}{z},0,\ldots)
\in C^{\N}\subseteq\wtB.
\end{equation}

\begin{lemma}[Escape in the product topology]\label{lem:e-k-zero}
The sequence $(e^{(k)})_{k\ge1}$ satisfies
\[
e^{(k)}\longrightarrow0
\quad\text{in }\wtB.
\]
\end{lemma}

\begin{proof}
Fix $m\ge1$.  For every $k>m$, the first $m$ coordinates of $e^{(k)}$ vanish.  Hence
\[
q_m(e^{(k)})=0
\qquad(k>m).
\]
Therefore $q_m(e^{(k)})\to0$ for every $m$, which is exactly convergence in $\wtB$.
\end{proof}

\begin{lemma}[Persistent noninvertibility]\label{lem:e-k-nonunit}
For every $k\ge1$,
\[
1+e^{(k)}\notin\wtB^{\times}.
\]
\end{lemma}

\begin{proof}
Apply the coordinate homomorphism $\rho_k$ from \eqref{eq:rho-j}.  By \eqref{eq:e-k} and \eqref{eq:fixed-z},
\[
\rho_k(1+e^{(k)})=1+z\notin(C^{\#})^{\times}.
\]
If $1+e^{(k)}$ were invertible in $\wtB$, its image under the unital algebra homomorphism $\rho_k$ would be invertible in $C^{\#}$.  This contradiction proves the claim.
\end{proof}

\begin{theorem}[Failure of openness]\label{thm:failure-openness}
The unit group $\wtB^{\times}$ is not open in $\wtB$.  Consequently, $\wtB$ is not a continuous inverse algebra.
\end{theorem}

\begin{proof}
By Lemma~\ref{lem:e-k-zero},
\[
1+e^{(k)}\longrightarrow1.
\]
By Lemma~\ref{lem:e-k-nonunit},
\[
1+e^{(k)}\in\wtB\setminus\wtB^{\times}
\qquad(k\ge1).
\]
Thus every neighborhood of $1$ contains a noninvertible element.  Hence
\[
1\notin\operatorname{int}_{\wtB}(\wtB^{\times}),
\]
so $\wtB^{\times}$ is not open.
\end{proof}

\begin{proof}[Proof of Theorem~\ref{thm:main}]
Choose the Dixon pair $(C,N)$ from Theorem~\ref{thm:dixon}, and construct $B$ by \eqref{eq:def-D} and \eqref{eq:def-B}.  Theorem~\ref{thm:B-CIA} gives
\[
B\text{ is a Hausdorff metrizable locally }m\text{-convex CIA}.
\]
Propositions~\ref{prop:Btilde-complete} and \ref{prop:completion-formula} give
\[
\widehat B\cong\wtB
\quad\text{with }\wtB\text{ Fr\'echet and locally }m\text{-convex}.
\]
Proposition~\ref{prop:Btilde-inversion} gives
\[
\inv_{\widehat B}\text{ is continuous on }\widehat B^{\times}.
\]
Theorem~\ref{thm:failure-openness} gives
\[
\widehat B^{\times}\text{ is not open}.
\]
Therefore $\widehat B$ is not a CIA.
\end{proof}

\section{Consequences and the exact failure mechanism}

The preceding construction separates the algebraic, topological, and spectral roles of the ingredients.  We now record the direct consequence for the completion problem and identify the boundary of the argument.

\subsection{The completion problem}

\begin{corollary}[Negative answer to the CIA completion problem]\label{cor:problem}
The completion of a continuous inverse algebra need not be a continuous inverse algebra.  Hence Problem~10.6.6 of \cite{GN2026} and Problem~VIII.1 of \cite{NeebSurvey} have a negative answer.
\end{corollary}

\begin{proof}
The algebra $B$ in Theorem~\ref{thm:main} is a CIA, whereas its completion $\widehat B$ is not a CIA.
\end{proof}

The example lies in a restrictive class on both sides:
\[
\begin{aligned}
&B&&\text{is metrizable and locally }m\text{-convex},\\
&\widehat B&&\text{is Fr\'echet and locally }m\text{-convex}.
\end{aligned}
\]
Thus failure is not caused by nonmetrizability, incompleteness of the completed algebra, or discontinuity of multiplication.

\subsection{Finite support versus product completion}

The proof can be summarized by the following chain:
\begin{equation}\label{eq:mechanism-chain}
\begin{aligned}
&d\in N^{(\N)}
\Longrightarrow
|\supp(d)|<\infty
\Longrightarrow
(\exists r(d))\ d^{r(d)}=0,\\
&d^{r(d)}=0
\Longrightarrow
\lambda1+d\in B^{\times}
\Longleftrightarrow
\lambda\ne0,\\
&\overline{N}^{\,C}=C
\Longrightarrow
\widehat{N^{(\N)}}=C^{\N},\\
&1+z\notin(C^{\#})^{\times}
\Longrightarrow
1+e^{(k)}\notin\widehat B^{\times},\\
&e^{(k)}\to0
\Longrightarrow
1+e^{(k)}\to1.
\end{aligned}
\end{equation}
The first two lines make $B$ a CIA.  The last three lines destroy openness after completion.

\subsection{Why the construction is noncommutative}

Gl\"ockner and Neeb prove that the completion of a commutative CIA is again a CIA \cite[Proposition~10.6.10]{GN2026}.  The Dixon input is necessarily outside that case.  Indeed, in a commutative Banach algebra every nilpotent element belongs to the Jacobson radical.  If $C$ were both commutative and Jacobson-semisimple, then
\[
N\subseteq\rad(C)=0,
\]
so density of $N$ would imply $C=0$, contrary to Theorem~\ref{thm:dixon}.  The noncommutativity is therefore not cosmetic; it is exactly where dense nil behavior can coexist with a semisimple completion.

\subsection{What survives completion}

The exact unit formula \eqref{eq:Btilde-units} and Proposition~\ref{prop:Btilde-inversion} show that completion preserves more structure than the final negative conclusion alone records:
\[
\begin{aligned}
&\text{multiplication extends continuously},\\
&\text{the completed algebra is Fr\'echet and locally }m\text{-convex},\\
&\text{the full algebraic unit group is explicitly known},\\
&\text{inversion on that unit group is continuous}.
\end{aligned}
\]
Only the local spectral condition at the identity fails:
\[
1\in\overline{\wtB\setminus\wtB^{\times}}.
\]
Thus the obstruction is precisely an infinite-coordinate loss of spectral openness.

\begin{remark}[Abstract form of the construction]\label{rem:abstract}
The proof uses from $C$ only the following two inputs:
\[
\begin{aligned}
&N\subseteq C\text{ is a dense nil subalgebra},\\
&(\exists z\in C)\qquad 1+z\notin(C^{\#})^{\times}.
\end{aligned}
\]
Whenever a Banach algebra $C$ and a dense nil subalgebra $N$ satisfy these conditions, the same finite-support product construction yields a CIA whose completion has nonopen unit group.  Dixon's theorem provides a canonical existence input through $C\ne0$ and $\rad(C)=0$.
\end{remark}

\section*{Acknowledgments and disclosure of AI assistance}

The author thanks the DeepMath team for its agent support. The initial ideas
and inspiration for this work came from the author and were further developed
through discussions with DeepMath agents. GPT models were used to carry out
the constructions, generate the manuscript text, and support adversarial review in
multiple separate conversations. The author has manually reviewed the present
manuscript and currently considers its arguments sound.

AI-assisted checking of the present version has been completed. Additional
independent line-by-line verification by human mathematicians is in progress, and a
separately written human-authored manuscript is also in preparation. The
mathematicians carrying out this additional verification are not authors of the present
version and are not included in its author list. After the verification is completed,
those who have made substantive contributions, approve the resulting manuscript,
and agree to take responsibility for its contents will be added to the author list in
a later version. The listed author takes full responsibility for all contents of the
present version.

\end{document}